\documentclass{amsart}

\usepackage{amsmath,amssymb,amsbsy,amsfonts, latexsym, amsopn,amstext, amsxtra,euscript,amscd}

\usepackage[colorlinks=true, urlcolor=blue, pdfborder={0 0 0}]{hyperref}

\usepackage{xcolor}

\usepackage{amsrefs}

\usepackage{url}

\newcommand\myurl[1]{\url{#1}}

\BibSpec{webpage}{%
  +{}{\PrintAuthors} {author}
  +{,}{ \textit} {title}
  +{}{ \parenthesize} {date}
  +{,}{ \myurl} {myurl}
  +{,}{ } {note}
  +{.}{ } {transition}
}

\usepackage{xy} 
\input xy \xyoption{all}

\newtheorem{thm}{Theorem}
\newtheorem{prop}{Proposition}
\newtheorem{lemma}{Lemma}
\newtheorem{rem}{Remark}
\newtheorem{cor}{Corollary}
\newtheorem{exa}{Example}

\def\Z{\mathbb Z}
\def\Q{\mathbb Q}
\def\C{\mathbb C}

\def\Q{{\mathbb Q}}
\def\Z{{\mathbb Z}}
\def\C{{\mathbb C}}

\def\L{\mathcal L}
\def\H{\mathcal H}

\def\X{\mathcal X}

\def\s{\mathfrak s}

\def\D{\Delta}

\def\emb{\hookrightarrow}

\def\a{\alpha}
\def\e{\varepsilon}

\def\t{\tau}
\def\<{\langle}
\def\>{\rangle}
\def\zz{\zeta}
\def\z{\omega}

\def\iso{{\, \cong\, }}

\def\G{G}

\def\Z{\mathbb Z}

\def\Q{\mathbb Q}

\def\C{\mathbb C}

\def\C{\mathcal C}
\def\H{\mathcal H}

\def\a{\alpha}

\def\Z{\mathbb Z}
\def\Q{\mathbb Q}
\def\C{\mathbb C}
\def\L{\mathcal L}
\def\H{\mathcal H}

\def\s{\sigma}
\def\a{\alpha}

\def\<{\langle}
\def\>{\rangle}
\def\X{\mathcal X}
\def\emb{\hookrightarrow}

\def\D{\Delta}

\def\L{\mathcal L}

\def\Q{{\mathbb Q}}
\def\Z{{\mathbb Z}}
\def\C{{\mathbb C}}

\def\H{\mathcal H}
\def\L{\mathcal L}

\def\O{\Omega}

\def\Aut{\mbox{Aut }}
\def\bAut{\overline {\mathrm{Aut}}}

\def\D{\Delta}
\def\e{\varepsilon}

\def\t{\tau}

\def\a{{\alpha }}

\def\<{\langle}
\def\>{\rangle}

\def\s{\mathfrak s}

\def\emb{\hookrightarrow }

\def\zz{\zeta}

\def\z{\omega}

\hypersetup{
  colorlinks   = true, 
  urlcolor     = blue, 
  linkcolor    = blue, 
  citecolor   = red 
}

\def\div{\mbox{div}}
\def\ord{\mbox{ord}}

\title[$2$-Weierstrass points of genus 3 curves]{The $2$-Weierstrass points of genus 3 hyperelliptic curves with extra automorphisms}

\author{T. Shaska}

\address{Department of Mathematics \& Statistics, Oakland University, Rochester, MI, 48309.}

\author{C. Shor}

\address{Department of Mathematics\\ Western New England University\\ Springfield, MA 01119.}


\date{\today}


\begin{document}

\begin{abstract} For each group $G$,  $(|G| > 2)$ \,  which acts as a full automorphism group on a genus 3 hyperelliptic curve, we determine the family of curves which have  $2$-Weierstrass points.  Such families of curves are explicitly determined in terms of the absolute invariants of binary octavics.   The 1-dimensional families that we discover have the property that they contain only genus 0 components. 
\end{abstract}

\maketitle

\section{Introduction}
The Weierstrass points of curves have always been a focus of investigation. The Riemann-Roch theorem shows that every point on a genus $g\geq 2$  curve has a non-constant function associated to it which has a pole of order less than or equal to $g+1$ and no other poles.  A \textit{Weierstrass point} is a point such that there is a non-constant function which has a low order pole and no other poles. By low order we mean a pole of order $\leq g$.  

Hurwitz showed that all Weierstrass points on a given curve are zeroes of a certain high order differential form.  The \textit{Weierstrass weight} of a point is the order of the zero of this form at the point. 
Since this differential form has degree $g^3-g$ then there are only finitely many Weierstrass points.  Moreover, these points are all algebraic over the field of definition of the curve; see \cite{silverman}.  When such points are rational and the curve is defined over $\Q$ then interesting applications arise in number theory. Bhargava and Gross showed in \cite{bhargava} that
when all hyperelliptic curves of genus $g \geq 2$  having a rational Weierstrass point (of weight 1) 
are ordered by height, the average size of the 2-Selmer group of their Jacobians is equal to 3.
As a consequence, using the Chabauty's method they  show that   the majority of hyperelliptic
curves of genus $g \geq 3$  with a rational Weierstrass point have fewer than 20 rational points.


This paper has as a starting point the paper by Farahat/Sakai \cite{FS} who classify the 3-Weierstrass points on genus two curves with extra involutions. They describe such points using the dihedral invariants of such curves as defined in \cite{sh-v}.  In \cite{b-th} the automorphism groups of genus 3 hyperelliptic curves are characterized in terms of the dihedral invariants.  Naturally one asks if the methods in \cite{FS} can be extended to genus 3 via methods described in \cite{b-th}. 

The goal of this paper is to classify the $q$-Weierstrass points of genus 3 hyperelliptic curves with extra automorphisms in terms of the coordinates of the hyperelliptic moduli $\H_3$.  
We fix a group $G$ which acts on a genus 3 hyperelliptic curve as a full automorphism group.  For a given signature, the locus of curves with automorphism group $G$ is an irreducible locus in the hyperelliptic moduli $\H_3$.  We only focus on the groups $G$ which determine a dimensional $d>0$ family in $\H_3$.  A complete list of such groups, signatures, and inclusions among the loci is given in \cite{b-th}.   The locus of curves $C$ with the Klein viergrouppe $V_4 \emb \Aut (C)$ is a 3-dimensional locus in $\H_3$. The appropriate invariants in this case are the dihedral invariants $\s_2, \s_3, \s_4$ as defined in \cite{b-th}. All the other cases can be described directly in terms of absolute invariants $t_1, \dots , t_6$ or their equivalents as defined in \cite{hyp_mod_3}.  The main goal of this paper is to study $q$-Weierstrass points of curves in each $G$-locus, for each $G$.

In the second section we give some basic preliminaries for Weierstrass points and their weights. Most of the material can be found everywhere in the literature. In particular we refer to \cite{stichtenoth, farkas, cornalba1}.  The third section is also an introduction to genus 3 hyperelliptic curves with extra involutions and their dihedral invariants.  Most of the material from that section can be found in \cite{b-th}. The dihedral invariants $\s_2, \s_3, \s_4$ which are defined in this section will be used in later sections to classify the Weierstrass points of genus 3 curves with extra automorphisms. 

In section four we focus on the $q$-Weierstrass points of genus 3 curves for $q=1$ and $2$. We show how to construct of basis for the space of holomorphic $q$-differentials. The main result of this section is the following. Let $f(x)=(x^8+ax^6+bx^4+cx^2+1)^{1/2}$,   $C$ be given by $y^2=f(x)^2$, and let $P_m^w$ be a finite non-branch point of $C$.  Let $N=\min\left\{n\in\mathbb{N} : n\geq 5, f^{(n)}(w)\neq0\right\}$. 
Then $P_m^w$ has $2$-weight $N-5$.  In particular, $P_m^w$ is a $2$-Weierstrass point if any only if $f^{(5)}(w)\neq0$.

In section five we compute the Wronskian $\Omega_q$ for $q=2$ in terms of coordinates in the hyperelliptic moduli for each case when the group $| G | >2$ and the $G$-locus is $> 0$ dimensional. Computations are challenging, especially in the case of $G \iso V_4$. In this case we make use of the dihedral invariants $\s_2, \s_3, \s_4$ which make such computations possible. Such invariants can be expressed in terms of the absolute invariants $t_1, \dots , t_6$ of binary octavics, which uniquely determine the isomorphism class of a genus 3 hyperelliptic curve; see \cite{hyp_mod_3} for details. 
Computations of the Wronskian is made easier by Lemma \ref{resultants}, which is a technical result of resultants for decomposable polynomials.  Such result is also helpful for all hyperelliptic curves when the reduced automorphism group is $| \bar G | > 2$; details will be explained in \cite{SS2}. 

A natural question is to extend methods of this paper to all hyperelliptic curves with extra automorphisms.  It seems as everything should follow smoothly as in the case of genus 3, other than the fact that computations will be more difficult.  The dihedral invariants of genus $g\geq 2$ hyperelliptic curves are defined in \cite{g-sh}.  A basis for the space of holomorphic differentials is known how to be constructed.  However, the computational aspects seem to be quite difficult.  This is studied in \cite{SS2}.   In the case of superelliptic curves (i.e. curves with equation $y^n = f(x)$) there has been some work done by Towse, see \cite{tw1, tw3}. A complete treatment of Weierstrass points of such curves with extra automorphisms  via their dihedral invariants is intended in \cite{SS3}. In both cases, the computation of the Wronskian is speeded up by Lemma \ref{resultants}.

By a curve we always mean a smooth, irreducible, algebraic curve defined over $\C$ or equivalently a compact Riemann surface.  A  curve $C$ will mean  the isomorphism class of $C$ defined over the field of complex numbers $\C$. Unless otherwise noted, a curve $C$ will denote a genus 3 hyperelliptic curve defined over $\C$.  

\section{Preliminaries}\label{sec:prelims}
Below we give the basic definitions of Weierstrass points  and establish some of the basic facts about the $q$-Weierstrass points of curves.

\subsection{Weierstrass points}
Following the notation of \cite{stichtenoth}, let $k$ be an algebraically closed field,  $C$ be a non-singular projective curve over $k$ of genus $g$, and  $k(C)$ its function field.  For any $f\in k(C)$,  $\div (f)$ denotes the divisor associated to $f$, $\div (f)_0$ and $\div (f)_\infty$ respectively the zero and pole divisors of $f$.  For any divisor $D$ on $C$, we have $D=\sum_{P\in C} n_P P$ for  $n_P \in \Z$ with almost all $n_P=0$.  Let $v_P(D)=n_P$, and let $v_P(f)=v_P(\div (f))$.

For any divisor $D$ on $C$, let $\L(D)=\{ f\in k(C) : \div (f)+D\geq0\}\cup\{0\}$ and $\ell(D)= \dim_k(\L(D))$.  By Riemann-Roch theorem, for any canonical divisor $K$, we have 
\[ \ell(D)-\ell(K-D)=\deg(D)+1-g.\]
Since the degree of a canonical divisor is $2g-2$, and since $\L(D)=\{0\}$ for any divisor $D$ with negative degree, if $\deg(D)\geq2g-1$, then $\deg(K-D)<0$, so $\ell(K-D)=0$.  Thus, if $\deg(D)\geq 2g-1$, then 
\[ \ell(D)=\deg(D)+1-g.\]
Let $P$ be a degree 1 point on $C$.  Consider the chain of vector spaces $$\L(0)\subseteq\L(P)\subseteq\L(2P)\subseteq\L(3P)\subseteq \dots \subseteq\L\left((2g-1)P\right).$$  Since $\L(0)=k$, we have $\ell(0)=1$.  And $\ell\left((2g-1)P\right)=g$.   We obtain the corresponding non-decreasing sequence of integers $$ \ell(0)=1, \ell(P), \ell(2P), \ell(3P), \dots, \ell\left((2g-1)P\right)=g.$$  It is straightforward to show that $0\leq \ell(nP)-\ell((n-1)P)\leq 1$ for all $n\in\mathbb{N}$.  If $\ell(nP)=\ell((n-1)P)$, then we call $n$ a \textit{Weierstrass gap number}.  For any point $P$, there are exactly $g$ Weierstrass gap numbers.  If the gap numbers are $1,2,\dots,g$, then $P$ is an \textit{ordinary point}.  Otherwise, we call $P$ a \textit{Weierstrass point}.  (Equivalently, we call $P$ a Weierstrass point if $\ell(gP)>1$.)

\subsection{$q$-Weierstrass points}
Using differentials, we can define $q$-Weierstrass points as in Chapter III. 5 of \cite{farkas}, as well as in \cite{FS}.  For any $q\in\mathbb{N}$, let $H^0(C,(\Omega^1)^q)$ be the $\mathbb{C}$-vector space of holomorphic $q$-differentials on $C$.    Let $d_q=\dim (H^0(C(\Omega^1)^q))$.  As an application of Riemman-Roch, by Proposition III.5.2 in \cite{farkas}, for $g\geq2$, one has 
\begin{equation}\label{d_q}
d_q = 
\begin{cases} 
g           & \text{if } q=1 \\ 
(g-1)(2q-1) & \text{if } q\geq2
\end{cases}
\end{equation}

As before, let $P$ be a degree 1 point on $C$.  Take a basis $\{\psi_1,\dots,\psi_{d_q}\}$ of $H^0(C,(\Omega^1)^q)$ such that 
\[ \ord_P(\psi_1)<\ord_P(\psi_2)<\cdots<\ord_P(\psi_{d_q}).\]
This can always be done, as in \cite[Section III.5]{farkas}.  For $i=1,\dots,d_q$, let $n_i=\ord_P(\psi_i)+1$.  The sequence of natural numbers $G^{(q)}(P)=\{n_1,n_2,\dots,n_{d_q}\}$ is called the \textit{$q$-gap sequence of $P$}.  With such a gap sequence, we can calculate the \textit{$q$-weight} of $P$, which is $$w^{(q)}(P)=\sum_{i=1}^{d_q}(n_i-i).$$  We call the point $P$ a \textit{$q$-Weierstrass point} if $w^{(q)}(P)>0$.

\def\W{\mathcal W}

Let $\W (C)$ denote the set of all Weierstrass points  and   $W_q(C)$  the set of all $q$-Weierstrass points on $C$.  In particular, $W_1(C)$, the set of $1$-Weierstrass points on $C$, is exactly the set of Weierstrass points as defined above in terms of Riemann-Roch. We summarize some properties in the following lemma.  

\begin{lemma} Let $C$ be a genus $g\geq 2$ curve.  The following hold:

i) There are $q$-Weierstrass points for any $q\geq 1$. 

ii)  For $q> 1$ 
\[ \sum_{P\in C} w^{(q)} (P) = g (g-1)^2 (2q-1)^2 \]

iii) $ 2g+2 \leq | W_1 (C) | \leq g^3-g $  
\end{lemma}

\proof The proofs can be found in \cite[Section III.5]{farkas}.  In particular, we know that 

%
\[ \sum_{P\in C} w^{(q)} (P) = (g-1) d (2q -1 +d) \]
where $d=d_q$ as in Eq.~\eqref{d_q}. Substituting for $d$ we get the result as claimed in ii). 

\qed

\noindent Now we give some results specific to the $g=3$ case.

\begin{exa}[Genus 3 curves]
For $g=3$ we have $d_q=2 (2q-1)$. The total weight  is 24 for $q=1$ and for $q> 1$ is 
\[ \sum_{P\in C} w^{(q)} (P) = 12 (2q-1)^2.\]
Notice that for $q=2$ we have $d_2= 6$ and the total weight is 108.  For  $q=3$, $d_3= 10$ and  the total weight is 300. In these cases we have, respectively, a $6\times6$ and a $10 \times 10$ Wronskian.
\end{exa}

\noindent In Section~\ref{sec:classification-q-w-points}, we give the following result for $q=2$,  cf. Remark 3. 

\begin{rem} Let $C$ be a genus 3 hyperelliptic curve.  
For  any point $P \in C$,  the $2$-weight of $P$ is $w^{(2)}(P)\leq6.$  Further, if $w^{(2)}(P)=6$, then $P\in W_1(C)$.  If $P\notin W_1(C)$, then $w^{(2)}(P)\leq3$.
\end{rem}

\subsection{The Wronskian}
Given a basis $\{\psi_1,\dots,\psi_{d_q}\}$ of $H^0(C,(\Omega^1)^q)$, where $\psi_i=f_i(x)dx$ for a holomorphic function $f_i$ of a local coordinate $x$ for each $i$, the \textit{Wronskian} is the determinant of the following $d_q\times d_q$ matrix:
$$W=\left|
\begin{matrix}
f_1(x) & f_2(x) & \cdots & f_{d_q}(x) \\
f_1'(x) & f_2'(x) & \cdots & f_{d_q}'(x) \\
\vdots & \vdots & \ddots & \vdots \\
f_1^{(d_q-1)}(x) & f_2^{(d_q-1)}(x) & \cdots & f_{d_q}^{(d_q-1)}(x)
\end{matrix}
\right|.$$  The Wronskian form is $\Omega_q=W(dx)^m$, for 
\begin{center}
$\begin{array}{rcl}
m&=&q+(q+1)+(q+2)+\dots+(q+d_q-1) \\ &=& (d_q/2)(2q-1+d_q).
\end{array}$
\end{center}
The following holds and its proof can be found everywhere in the literature. :
\begin{rem}
$P$ is a $q$-Weierstrass point with weight $w^{(q)} (P)= r$ if and only if  $P$ is a zero of multiplicity $r$ for the differential form $\Omega_q$ or equivalently in the support of $\div (\O_q)$.
\end{rem}

Since the Wronskian form is a holomorphic $m$-differential, div$(\Omega_q)$ is effective.  Thus, the $q$-Weierstrass points are the support of $\div (\Omega_q)$, and the sum of the $q$-weights of the $q$-Weierstrass points is the degree of  div$(\Omega_q)$, which is $m(2g-2)=d_q(2q-1+d_q)(g-1)$.  In particular, this means there are a finite number of $q$-Weierstrass points.




\section{Genus 3 hyperelliptic fields with extra automorphisms}

Let $K$ be a genus 3 hyperelliptic field. Then $K$ has exactly one genus 0 subfield of degree 2, call it $k(X)$.  It is the fixed field of the \textbf{hyperelliptic involution} $\z_0$ in $\Aut (K)$. Thus, $\z_0$ is central in $\Aut (K)$, where  $\Aut (K)$ denotes the group $\Aut (K/k)$. It induces a subgroup of $\Aut (k(X))$ which is naturally isomorphic to $\bAut (K):= \Aut (K)/\<\z_0\>$. The latter is called the \textbf{reduced automorphism group} of $K$.

An \textbf{extra involution} (or non-hyperelliptic)  of $G = \Aut (K)$  is  an  involution different from $\z_0$. Thus,  the extra involutions of $G$ are in  1-1 correspondence with the non-hyperelliptic subfields of $K$ of degree 2. 

Let  $\e$ be an extra involution in $\G$. We can choose the generator $X$ of $\mbox{Fix}(\z_0)$      such that $\e(X)=-X$.  Then $K=k(X,Y)$ where $X, Y$ satisfy equation
\begin{equation}\label{eq_1} Y^2 = (X^2-\a_1^2) (X^2-\a_2^2) (X^2-\a_3^2) (X^2-\a_4^2)  \end{equation}
for some $\a_i \in k$, $i=1, \dots, 4$.  Denote by
\begin{equation}
\begin{split}
s_1 = & - \left( \a_1^2 + \a_2^2 + \a_3^2 + \a_4^2 \right) \\
s_2 = & \,  (\a_1 \a_2)^2 + (\a_1 \a_3)^2 + (\a_1 \a_4)^2 + (\a_2 \a_3)^2 + (\a_2 \a_4)^2 + (\a_3 \a_4)^2 \\
s_3 = & -   ( \a_1\,\a_2\,\a_3 )^2 - (\a_4\,\a_1\,\a_2)^2 - (\a_4\,\a_3\,\a_1)^2 - (\a_4\,\a_3\,\a_2)^2 \\
s_4 = & - \left( \a_1 \a_2 \a_3 \a_4 \right)^2\\
\end{split}
\end{equation}
Then,  we have 
\[ Y^2=X^8+s_1 X^6+ s_2 X^4+ s_3 X^2+s_4  \]
with $s_1, s_2, s_3, s_4 \in k$, $s_4\ne 0$. Further $E_1=k(X^2,Y)$ and $ C=k(X^2, YX)$ are the two  subfields corresponding to $\e$ of genus 1 and 2 respectively.   

\par Preserving the condition $\e(X)=-X$ we  can further modify $X$ such that  $s_4=1$. Then, we have the following lemma, which  is proven in \cite{b-th}.

\begin{lemma}  Every genus 3 hyperelliptic curve $\X$,  defined over a field $k$, which has an non-hyperelliptic involution has equation 
\begin{equation}\label{eq3}
Y^2=X^8+a X^6+ b X^4+ c X^2+1
\end{equation}
for some $a, b, c \in k^3$,  where the polynomial on the right has non-zero discriminant.
\end{lemma}

The above  conditions  determine $X$ up to coordinate change by the group $\< \tau_1, \tau_2\>$ where  
$ \tau_1: X\to \zz_8 X$,  and $ \tau_2: X\to \frac 1 X$,  and  $\zz_8$ is a primitive 8-th root of unity in  $k$. 
Hence,  
\[\tau_1 : \, (a, b, c) \to \left(\zz_8^6 a,  \zz_8^4 b,  \zz^2 c \right) \quad \textit{ and } \quad  \tau_2 : \, (a, b,  c) \to \left(c, b, a \right).\]
Then, $| \tau_1 | =4$ and $|\tau_2 | =2$.  The group generated by $\tau_1$ and $\tau_2$ is the dihedral group of order 8.   Invariants of this action are
\begin{equation}\label{eq2}
\s_2   =\,a\,c,  \quad  \s_3   =(a^2+c^2)\,b,  \quad   \s_4   =a^4+c^4, 
\end{equation}
since 
\[ 
\begin{split}
& \tau_1 (a^4+c^4) = (\zz_8^6 a)^4 + (\zz_8^2 c)^4 = a^4 + c^4 \\
& \tau_1 \left(   (a^2 + c^2 ) b \right) = \left( \zz_8^4 a^2 + \zz_8^4 c^2   \right) \cdot ( \zz_8^4 b) = (a^2+c^2) b \\
& \tau_1 ( a c ) =  \zz_8^6 a \cdot \zz_8^2 c =  a c \\
\end{split}
\] 
Since they are symmetric in $a$ and $c$, then they are obviously invariant under $\t_2$. 
Notice that $\s_2, \s_3, \s_4$ are homogenous polynomials of degree 2, 3, and 4 respectively.  The subscript $i$ represents the degree of the polynomial $\s_i$.

Since the above transformations are automorphisms of the projective line $\mathbb P^1 (k)$ then the $SL_2(k)$ invariants must be expressed in terms of $\s_4, \s_3$, and $\s_2$. 
In these parameters, the discriminant $\D (\s_2, \s_3, \s_4) $  of the octavic polynomial on the right hand side of  Eq.~ \eqref{eq3} is nonzero.

 The map \[(a, b, c) \mapsto (\s_2, \s_3, \s_4)\] is a branched Galois covering with group $D_4$  of the set 
\[ \{  (\s_2, \s_3, \s_4)\in k^3 : \Delta_{(\s_2, \s_3, \s_4)} \neq 0\}\]
by the corresponding open subset of $a, b, c$-space. In any case, it is true that if $a, b, c$ and $a', b', c'$ have the same $\s_2, \s_3, \s_4$-invariants then they are conjugate under $\< \tau_1, \tau_2\>$.


Let $C$ be a genus 3 hyperelliptic curve defined over $\C$, $K$ its function field, and $G$ be the full automorphism group  $G:= \Aut (K)$. The loci of genus 3 hyperelliptic  curves $C$ with full automorphism group $G$ are studied in \cite{g-sh-s, b-th}.  All such groups $G$ have distinct ramification structures and therefore there is no confusion to denote such locus $\H(G)$ for any fixed $G$.
In this paper we will make use of the following facts, which are proven in \cite{b-th}. 

\begin{thm} Let $C$ be a genus 3 hyperelliptic curve such that $| \Aut (C) | > 2$ and $\dim \H \left(\Aut (C) \right) \geq 1$. Then, one of the following holds:

i) $\Aut (C) \iso V_4$ and the locus $\H (V_4)$ is 3-dimensional.  A generic curve in this locus has equation 
\begin{equation} y^2 = A \, x^8 + \frac {A} {\s_4 + 2\s_2^2} \, \, x^6  +  \frac {\s_3  ( A + \s_2^2) } {( \s_4 + 2\s_2^2 )^3} \, \, x^4  +  \frac { \s_2} {(\s_4 + 2\s_2^2)^3} \, \, x^2 + \frac {1} {(\s_4 + 2\s_2^2)^4} \end{equation}
where $A$ satisfies   $A^2 -  \s_4 A + \s_2^4 =0$. 

ii) $\Aut (C) \iso \Z_2^3 $ and the locus $\H (\Z_2^3)$ is 2-dimensional.  A generic curve in this locus has equation 
\begin{equation} y^2= \s_2^2 x^8 + \s_2^2 x^6 + \frac 1 2 \, \s_3 x^4 + \s_2 x^2 + 1. \end{equation}

iii) $\Aut (C) \iso \Z_2\times D_8 $  and the locus $\H (\Z_2 \times D_8)$ is 1-dimensional.  A generic curve in this locus has equation 
\begin{equation} y^2= t x^8 + t x^4 +1. \end{equation}

iv) $\Aut (C) \iso D_{12} $  and the locus $\H (D_{12})$ is 1-dimensional.  A generic curve in this locus has equation 
\begin{equation} y^2= x \, ( tx^6 + t x^3 +1) \end{equation}

v) $\Aut (C) \iso \Z_2 \times \Z_4 $ and the locus $\H (\Z_2 \times \Z_4)$ is 1-dimensional.  A generic curve in this locus has equation  
\begin{equation} y^2 = \left( t x^4-1 \right) \, \left( t x^4 + tx^2+1\right) \end{equation}

\end{thm}

\begin{rem} i) Notice that in each case of the above Theroem, it is assumed that the discriminant of the polynomial in $x$ is not zero.

ii)  Other than the case i) in all other cases the field of moduli is a field of definition. The equations provided in the above Theorem give a rational model of the curve over its field of moduli. In other words, the dihedral invariants $\s_2 \s_3, \s_4$ or $t$ are uniquely determined as rational functions in terms of the absolute invariants $t_1, \dots , t_6$. 
\end{rem}


\section{Classification of $2$-Weierstrass points for genus 3 hyperelliptic curves}\label{sec:classification-q-w-points}

Let $C$ be a hyperelliptic curve of genus $g=3$ with non-hyperelliptic involution, as in Eq.~\eqref{eq3}.  As in Eq.~\eqref{eq_1}, let $\{\pm\alpha_1, \pm\alpha_2, \pm\alpha_3, \pm\alpha_4\}$ denote the 8 distinct roots of $f(x)$, and denote the corresponding ramification points on $C$ by $R_{i}^\pm=(\pm\alpha_i,0)$.  Throughout this section, let $w\in\mathbb{C}$ denote any non-root of $f(x)$, and let $P_1^w$ and $P_2^w$ denote the two (distinct) points above $w$.  And let $P_1^\infty$ and $P_2^\infty$ denote the two points over $\infty$ in the non-singular model of $C$.

Here are the divisors associated to the differential $dx$ and some functions:
\bigskip

\begin{center}
$\begin{array}{|l|l|} \hline & \\
\div (dx)=\displaystyle\left(\sum_{i=1}^4 R_i^{\pm}\right) - 2(P_1^\infty+P_2^\infty) & 
\div (y)=\displaystyle\left(\sum_{i=1}^4 R_i^{\pm}\right) - 4(P_1^\infty+P_2^\infty) \\ & \\ \hline & \\
\div (x- (\pm \alpha_i)) = 2R_i^\pm - (P_1^\infty+P_2^\infty) &
\div (x-w) = P_1^w+P_2^w - (P_1^\infty+P_2^\infty) \\ & \\ 

 \hline
\end{array}$\end{center}\bigskip

Consider $H^0(C,(\Omega^1)^q)$, the space of holomorphic $q$-differentials on $C$.  For a curve of genus $g$, by Riemann-Roch one has that $\dim(H^0(C,(\Omega^1)^1))=g$ and, for $g\geq2$ and $q\geq2$, 
\[\dim(H^0(C,(\Omega^1)^q))=(g-1)(2q-1).
\]
In particular, for $g=3$, when $q\geq2$, $d_q=\dim(H^0(C,(\Omega^1)^q))=2 (2q-1).$

\begin{thm}\label{thm:basis-for-q-differentials}
Let $C$ be a hyperelliptic curve of genus $g=3$ given by the equation $y^2=f(x)$ with $\deg(f(x))=8$.  Then one has the following bases of holomorphic $q$-differentials.

For $q=1$, a basis for $H^0(C,(\Omega^1)^1)$ is
\[
B_{1,\beta}=\left\{\frac{(x-\beta)^j}{y}dx : 0\leq j\leq 2\right\}
\]
for any $\beta\in\mathbb{C}$.

For $q=2$, a basis for $H^0(C,(\Omega^1)^2)$, is
\[
B_{2,\beta}=\left\{\frac{(x-\beta)^j}{y^2}(dx)^2 : 0\leq j\leq 4 \right\} \cup \left\{\frac{(y-f_{\beta,4,m}(x))}{y^2}(dx)^2\right\}
\]
for any $\beta\in\mathbb{C}$ and $m\in\{1,2\}$, with $f_{\pm\alpha_i,4,m}(x)=0$ and $f_{w,4,m}(x)$ the degree-$4$ Taylor polynomial for $C$ at $P_m^w$.

For $q\geq2$, a basis for $H^0(C,(\Omega^1)^q)$, is
\[
B_{q,\beta}=\left\{\frac{(x-\beta)^j}{y^q}(dx)^q : 0\leq j\leq 2q \right\} \cup \left\{\frac{(x-\beta)^jy}{y^q}(dx)^q : 0\leq j\leq 2q-4  \right\}
\]
for any $\beta=\pm\alpha_i$.
\end{thm}

\begin{proof}
A simple count shows that these bases have the correct number of elements.  To prove this theorem, we must show that the basis elements are linearly independent.  It will suffice to show that the $q$-differentials are holomorphic and have different orders of vanishing at at a particular point.

For $q=1$, the divisors associated to the $1$-differentials in $B_{1,\beta}$ are:
\begin{itemize}
\item $\div \left(\frac{(x-\alpha_i)^j}{y}dx\right)=2jR_i^++(2-j)(P_1^\infty+P_2^\infty)$,
\item $\div \left(\frac{(x+\alpha_i)^j}{y}dx\right)=2jR_i^-+(2-j)(P_1^\infty+P_2^\infty)$,
\item $\div \left(\frac{(x-w)^j}{y}dx\right)=j(P_1^w+P_2^w)+(2-j)(P_1^\infty+P_2^\infty)$,
\end{itemize} for $0\leq j\leq2$.
For any $\beta$, the $1$-differentials in $B_{1,\beta}$ are holomorphic and have zeros of degree 2, 1, 0 at the points at infinity.  Hence, $B_{1,\beta}$ is a basis for $H^0(C,(\Omega^1)^1).$

\bigskip

For $q=2$, $B_{2,\beta}$ is the union of two sets.  We consider two cases: $\beta=\pm\alpha_i$ and $\beta=w$.

For $\beta=\pm\alpha_i$, the divisors associated to the $2$-differentials with in the first set are
\begin{itemize}
\item $\div \left(\frac{(x-\alpha_i)^j}{y^2}(dx)^2\right)=2j(R_i^+)+(4-j)(P_1^\infty+P_2^\infty)$,
\item $\div \left(\frac{(x+\alpha_i)^j}{y^2}(dx)^2\right)=2j(R_i^-)+(4-j)(P_1^\infty+P_2^\infty)$,
\end{itemize}
for $0\leq j\leq 4$, and in the second set,
\begin{itemize}
\item $\div \left(\frac{y}{y^2}(dx)^2\right)=\sum_{i=1}^4(R_i^++R_i^-)$.
\end{itemize}

These $2$-differentials are holomorphic with orders of vanishing at $R_i^\pm$ equal to $0,2,4,6,8,$ and $1$.  Since these orders are all different, for $\beta=\pm\alpha_i$, $B_{2,\beta}$ is a basis for $H^0(C,(\Omega^1)^2)$.

For $\beta=w$, we have $f_{w,4,m}(x)$ the degree-$4$ Taylor polynomial for $C$ at $P_m^w$.  Let $D=\div (y-f_{w,4,m}(x))$.  By construction, $v_{P_m^w}(D)\geq 5$.  Also, for $i\in\{1,2\}$, since $v_{P_i^\infty}(y)=-4$, $v_{P_i^\infty}(f_{w,4,m}(x))\geq-4$, and $y\neq f_{w,4,m}(x)$, we conclude that $v_{P_i^\infty}(y-f_{w,4,m}(x))\geq-4$.  Then
\begin{itemize}
\item $\div \left(\frac{(x-w)^j}{y^2}(dx)^2\right)=j(P_1^w+P_2^w)+(4-j)(P_1^\infty+P_2^\infty)$,
\end{itemize}
for $0\leq j\leq 4$, and
\begin{itemize}
\item $\div \left(\frac{(y-f_{w,4,m}(x))}{y^2}(dx)^2\right)=D+4(P_1^\infty+P_2^\infty)$.
\end{itemize}
Since $(y-f_{w,4,m}(x))$ has poles only at $P_i^\infty$ with order at most 4, these $2$-differentials are all holomorphic.  The orders of vanishing at $P_m^w$ are $0, 1, 2, 3, 4$, and $v_{P_m^w}(D)\geq5$.  Since these orders are all different, for $\beta=w$, $B_{2,\beta}$ is a basis for $H^0(C,(\Omega^1)^2).$

\bigskip

Finally, for $q\geq2$, $B_{q,\beta}$ is the union of two sets.  We consider two cases: $\beta=\pm\alpha_i$ and $\beta=w$.

For $\beta=\pm\alpha_i$, the divisors associated to the $q$-differentials in the first set are
\begin{itemize}
\item $\div \left(\frac{(x-\alpha_i)^j}{y^q}(dx)^q\right)=2j(R_i^+)+(2q-j)(P_1^\infty+P_2^\infty)$,
\item $\div \left(\frac{(x+\alpha_i)^j}{y^q}(dx)^q\right)=2j(R_i^-)+(2q-j)(P_1^\infty+P_2^\infty)$,
\end{itemize}
for $0\leq j\leq 2q$, and in the second set are
\begin{itemize}
\item $\div \left(\frac{(x-\alpha_i)^jy}{y^q}(dx)^q\right)=\sum_{i=1}^4(R_i^++R_i^-)+2j(R_i^+)+(2q-j-4)(P_1^\infty+P_2^\infty)$,
\item $\div \left(\frac{(x+\alpha_i)^jy}{y^q}(dx)^q\right)=\sum_{i=1}^4(R_i^++R_i^-)+2j(R_i^-)+(2q-j-4)(P_1^\infty+P_2^\infty)$
\end{itemize}
for $0\leq j\leq 2q-4$.

These $q$-differentials are holomorphic with orders of vanishing at $R_i^\pm$ equal to $0,2,4,\dots,4q$ and $1,3,\dots, 4q-7$.  Since these orders are all different, for $\beta=\pm\alpha_i$, $B_{q,\beta}$ is a basis for $H^0(C,(\Omega^1)^q)$.
\end{proof}



\begin{cor}
For any $q\geq2$ and any branch point $R_i^\pm$, using $\beta=\pm\alpha_i$ in the basis $B_{q,\beta}$ as in Theorem~\ref{thm:basis-for-q-differentials}, the $q$-gap sequence is $\{1,2,3,\dots,4q-6,4q-5,4q-3, 4q-1, 4q+1\}$, so $w^{(q)}(R_i^\pm)=6$.  For $q=1$, the $1$-gap sequence is $\{1,3,5\}$, so $w^{(1)}(R_i^\pm)=3$.
\end{cor}

Hence, for $q\geq2$ the eight branch points contribute $8\cdot6=48$ to the total weight of $q$-Weierstrass points on the curve.  

In particular, $2$-gap sequence for a branch point is $\{1,2,3,5,7,9\}$.  The corollary below gives the $2$-gap sequence for a non-branch point.

\begin{cor}
Using $\beta=w$, the $2$-gap sequence of each finite non-branch point $P_m^w$ is $\{1,2,3,4,5,n_6\}$, with $n_6=v_{P_m^w}(y-f_{w,4,m}(x))+1$.  Therefore, if $n_6>6$, then $P_m^w$ is a $2$-Weierstrass point with $2$-weight $n_6-6$.
\end{cor}

\begin{rem}\label{rem:possible-2-gap-sequences}
Following from \cite{Duma}, the possible $2$-gap sequences of $2$-Weierstrass points on a curve of genus 3 are given in \cite[Lemma 5]{alwaleed}.  From this, we see that if $P_m^w$ is a non-branch point on a hyperelliptic curve of genus 3, the $2$-gap sequence contains 4 and 5, so $w^{(2)}(P_m^w)\leq3$.
\end{rem}

\begin{thm}\label{thm:2-w-points}
Let $C$ be a genus 3 hyperelliptic curve with equation $y^2=g(x)$ for $g(x)$ a separable degree 8 polynomial.  Let $f(x)=(g(x))^{1/2}$, and let $P_m^w$ be a finite non-branch point of $C$.  Let 
\[N=\min\left\{n\in\mathbb{N} : n\geq 5, f^{(n)}(w)\neq0\right\},\]
where $f^{(n)}(x)$ denotes the $n$th derivative of $f(x)$.  
Then $P_m^w$ has $2$-weight $N-5$.  In particular, $P_m^w$ is a $2$-Weierstrass point if any only if $f^{(5)}(w)\neq0$.
\end{thm}

\begin{proof}
Let
\[T_{5,w}(x)=\displaystyle\sum_{n=5}^\infty \frac{f^{(n)}(w)}{n!}(x-w)^n\]
be the difference of the Taylor series and degree-4 Taylor polynomial of $f(x)$ at $x=w$.  Then $v_{P_m^w}(y-f_{w,4,m}(x))=v_{P_m^w}(T_{5,w}(x))$, which is $N$, the degree of the first non-zero term in this series, so $N$ is the minimum value of $n\geq5$ such that $f^{(n)}(w)\neq0$.  Then $n_6=N+1$, so $P_m^w$ has $2$-weight $n_6-6=N-5$, which is positive when $N\geq6,$ or, equivalently, when $f^{(5)}(w)=0$.
\end{proof}

In particular, Remark~\ref{rem:possible-2-gap-sequences} implies that, for the value of $N$ in Theorem~\ref{thm:2-w-points}, one has $5\leq N\leq8$.

%
Computationally, to determine which curves have $2$-Weierstrass points we have to compute the Wronskian.  
Let $\psi_j=x^{j-1}/y^2$ for $1\leq j\leq5$ and $\psi_6=y/y^2$.  Then   $ \{\psi_j (dx)^2 : 1\leq j\leq6\}$ 
is a basis of holomorphic $2$-differentials.  Hence,   $\Omega_2 = W (1, x, x^2, x^3, x^4, y)$. 
Using Maple to compute the Wronskian $W$, we find $W=c_0\Phi(x)/y^{21}$ where $c_0$ is a constant and $\Phi(x)$ is a polynomial with $\deg(\Phi)\leq 29$ (depending on the values of $a,b,c$, described in more detail in Sec. \ref{sec:classification-2}).  The Wronskian form is $\Omega_2=W(dx)^{27}$.  Thus,
\begin{center}
$\begin{array}{rcl}
\div (\Omega_2) &=& \div (\Phi(x))-\div (y^{21})+\div ((dx)^{27}) \\
&=&\displaystyle\div (\Phi(x))_0 + 6\left(\sum_{i=1}^4(R_i^++R_i^-)\right)+(30-\deg(\Phi))(P_1^\infty+P_2^\infty).
\end{array}$
\end{center}
As expected, since $\deg(\Phi)\leq29$, $\div (\Omega_2)$ is an effective divisor.  Thus, the $2$-Weierstrass points are: the points $P_m^w$ over roots of $\Phi(x)$, with $2$-weight given by the order of vanishing; the branch points $R_i^\pm$ with $2$-weight 6; and the points at infinity with $2$-weight equal to $30-\deg(\Phi)$.

\section{Computation  of   $2$-Weierstrass points}\label{sec:classification-2}
In this section we will compute the Weierstrass points of weight $q=2$. We would like to determine suffiecient conditions for a curve to have $2$-Weierstrass points 
in each family $\H (G)$ such that $\dim \H (G ) > 0$.  

Usually such computations are difficult because of the size of the Wronskian matrix and the fact that its entries are polynomials.   Symbolic computational techniques have their limitations when dealing with large degree polynomials. However, in our case such computations are made possible by the use of the dihedral invariants and the results in \cite{b-th} and the following technical result.
\begin{lemma}\label{resultants}
Let $h(t)= \sum_{i=0}^n a_i \, t^{i}$ and $g(x)=h(x^2)$.  Then, 
\[ \D( g, x) = 2^{2n} \cdot a_0 a_n \cdot \D (h, t)^2  \]
\end{lemma}

The proof is elementary and is based on the definition of the resultant as the determinant of the Sylvester's matrix. We provide the details in \cite{SS2}.

\subsection{Genus 3 curves with extra involutions}

In this section we describe how to determine the curves with full automorphism group $V_4$ which have $2$-Weierstrass points.  The computational results are large to display. 

 Let $C$ be a genus 3 hyperelliptic curve such that $\Aut (C) \iso V_4$.  From \cite{b-th} we know that the equation of $C$ can be given by  $ y^2=x^8+ax^6+bx^4+cx^2+1$. 

In Theorem~\ref{thm:2-w-points}, we see that the $2$-Weierstrass points are the zeros of $f^{(5)}(x)$, where $f(x)=(x^8+ax^6+bx^4+cx^2+1)^{(1/2)}$.  The $2$-Weierstrass points are the roots of the polynomial $\Phi(x)$.  If we expand out $f^{(5)}(x)$, we get $f^{(5)}(x)=g(x)/f(x)^9$, where $g(x)$ is a constant multiple of $\Phi(x)$.  Using Maple, we find that 
\[ \Omega_2 = W (1, x, x^2, x^3, x^4, y) = 4320 \frac { x \, \Phi(x^2, a, b, c )} {\left(    \sqrt{x^8+a x^6+bx^4+cx^2+1}\right)^{21} }(dx)^{27},\]  
where $\deg (\Phi, x) = 28$.   Denote by  $\Phi (x) = \sum_{i=0}^{14} c_i x^{2i} $.  
Then  its coefficients   $c_i$ and $c_{14-i}$ differ by a permutation of  $a$ and  $c$.  In other words the permutation of the curve  $\tau_1 : (x, y) \to \left(\frac 1 x, y \right)$   acts on the coefficients of $\Phi (x)$ as follows
\[ \tau_1 (c_i) = c_{14-i} \]
Computing the discriminant $\Delta (\Phi, x)$ we get the following  factors as follows:
\[ \D = 2^{16} \, c_0 c_{14} \cdot g(a,b,c)^2 \cdot \Delta (f, x)^{28} \]
where $g(a,b, c)$ is a degree 24, 28, 24 polynomial in terms of $a, b, c$ respectively.  We know that $\D (f,x) \neq 0$. Let us assume that $c_0 c_{14} \neq 0$.  Then, the $2$-Weierstrass points are those when $g(a, b, c) =0$. The polynomial $g(a,b,c)$ can be easily computed.  However,  the triples $(a, b, c)$ do not correspond uniquely to the isomorphism classes of curves.  Naturally we would prefer to express such result in terms of the dihedral invariants $\s_2, \s_3, \s_4$.   One can take the equations $g(a,b,c)=0$ and three equations from the definitions of $\s_2, \s_3, \s_4$ and eleminate $a, b, c$.  It turns out that this is a challenging task computationally.  

Hence, we continue with the following approach.   From \cite[Prop. 2]{b-th} we know that $C$  is isomorphic to a curve with equation 
\[ y^2 = A \, x^8 + \frac {A} {\s_4 + 2\s_2^2} \, \, x^6  +  \frac {\s_3  ( A + \s_2^2) } {( \s_4 + 2\s_2^2 )^3} \, \, x^4  +  \frac { \s_2} {(\s_4 + 2\s_2^2)^3} \, \, x^2 + \frac {1} {(\s_4 + 2\s_2^2)^4} \]
where $A$ satisfies
\begin{equation}\label{quadratic}  A^2 -  \s_4 A + \s_2^4 =0, \end{equation}
for some $(\s_2, \s_3, \s_4 ) \in  k^3 \setminus  \{ \D_{\s_2, \s_3, \s_4 } = 0 \}$).

The Wronskian is a degree 29 polynomial in $x$ written as $x \, \phi(x^2)$.  From the above Lemma, it is enough to compute the discriminant of the polynomial $\phi (t)$, where $t=x^2$.  This is a degree 14 polynomial.  Its discriminant is a polynomial $G(A, \s_2, \s_3, \s_4)$ in terms of $\s_2, \s_3, \s_4$ and $A$.   Then, the relation between $\s_2, \s_3, \s_4$ is obtained by taking the resultant 
\begin{equation} \mbox{Res }(G, A^2-\s_4 A + \s_2^4, A).
\end{equation}
%

\noindent The result is quite a large polynomial in terms of $\s_2, \s_3, \s_4$. It turns out that the rest of the cases are much easier. 

\subsection{The case when    $\Aut (C) \iso \Z_2^3$.  }

\begin{prop}
Let $C$ be a genus 3 hyperelliptic curve with full automorphism group $\Z_2^3$. Then, $C$ has non-branch $2$-Weierstrass points of weight greater than one if and only if its corresponding dihedral invariants $\s_2, \s_3, \s_4$ satisfy Eq.~\eqref{eq_Z2_3}
%
\begin{equation}\label{eq_Z2_3}
\begin{split}
\Delta &  =   \left(  -784\,{\s_2}^2+16\,{\s_2}^3+56\,\s_2\s_3-{\s_3}^2 \right)\,  G(\s_2, \s_3) =0
\end{split}
\end{equation}
%
where 
\begin{Small}
\[ 
\begin{split}
G & =   617400 \s_3^9 +  180 s_2 \left(315560+871 s_2 \right)  \s_3^8   + 2 \s_2^2 \left(4023040 s_2+970717440+31077 \s_2^2 \right)  \s_3^7    \\
& + \s_2^3 \left(9 \s_2^3-3251241728 s_2+31937525760+ 5011112 \s_2^2 \right)  \s_3^6  + 8 \s_2^4 \left(-9204034560 s_2 \right. \\
& \left. -105048636 \s_2^2+15801 \s_2^3+41193015808 \right) \s_3^5  -16 \s_2^5 \left(22041513 \s_2^3+59872104320 s_2 \right. \\
& \left. +11 \s_2^4-4535327496 \s_2^2-193117539328 \right) \s_3^4  -256 \s_2^6 \left(-2870647262 \s_2^2-73789452800 \right. \\
& \left. +34876810752 s_2-47803959 \s_2^3+54199 \s_2^4 \right) \s_3^3  -256 \s_2^{7} \left(5 \s_2^5-41807037944 \s_2^2 \right. \\
& \left. +2624158985 \s_2^3-19769334 \s_2^4+283441853184 s_2-365995685888 \right)  \s_3^2 -2048 \s_2^{8} \left(308705831 \s_2^3 \right. \\
& \left. +28144998 \s_2^4-39227605228 \s_2^2+123966280704 s_2+31711 \s_2^5-175618897664 \right) \s_3    \\
& + 4096 \s_2^9 \left(455870765 \s_2^4-4058869 \s_2^5+7 \s_2^6-16649626455 \s_2^3-214358360384 s_2   \right. \\
& \left.   +85982595160 \s_2^2+144627327488 \right) \\
\end{split}
\]
\end{Small}

\end{prop}

\proof   From \cite[Lemma 6]{b-th} the equation of the curve is given by 
\[ y^2= \s_2^2 x^8 + \s_2^2 x^6 + \frac 1 2 \, \s_3 x^4 + \s_2 x^2 + 1,\]
for $s_2, \s_3 \neq 0, 4$.  The Wronskian is $\Omega_2 = W(1,x,x^2, x^3, x^4, y)(dx)^{27}$ as follows
\[  \Omega_2 = \frac {x \, ( \s_2 x^4-1 )} {4\,{s_{{2}}}^{2}{x}^{8}+4\,{s_{{2}}}^{2}{x}^{6}+2\,s_{{3}}{x}^{4}+4\,s_{{2}}{x}^{2}+4}   \, g(t) (dx)^{27}   \]
where $g(t) = \sum_{i=0}^{12} c_i \cdot t^i$ is a degree 12 polynomial for $t=x^2$ and    coefficients:
\[\begin{split}
c_{12} & = 12\,{\s_{{2}}}^{7} \\
c_{11} & = -4\,{\s_{{2}}}^{5} \left( -28\,\s_{{2}}+4\,{\s_{{2}}}^{2}-7\,\s_{{3}} \right)  \\
c_{10} & = 12\,{\s_{{2}}}^{5} \left( 22\,\s_{{2}}-3\,\s_{{3}} \right)  \\
c_{9} & = -4\,{\s_{{2}}}^{3} \left( 4\,{\s_{{2}}}^{3}+9\,{\s_{{3}}}^{2}-28\,{\s_{{2}}}^{2}+29\,\s_{{2}}\s_{{3}} \right)  \\
c_{8} & = {\s_{{2}}}^{3} \left( -{\s_{{3}}}^{2}-1180\,{\s_{{2}}}^{2}+16\,{\s_{{2}}}^{3}-376\,\s_{{2}}\s_{{3}} \right)  \\
c_{7} & = -\s_{{2}} \left( -3\,{\s_{{3}}}^{3}+24\,\s_{{2}}{\s_{{3}}}^{2}+48\,{\s_{{2}}}^{3}\s_{{3}}+1696\,{\s_{{2}}}^{4}+1568\,{\s_{{2}}}^{3}+536\,{\s_{{2}}}^{2}\s_{{3}} \right)  \\
c_{6} & = -26\,{\s_{{2}}}^{2} \left( 20\,\s_{{2}}\s_{{3}}+152\,{\s_{{2}}}^{2}-{\s_{{3}}}^{2}+16\,{\s_{{2}}}^{3} \right)  \\
c_{5} & = 3\,{\s_{{3}}}^{3}-24\,\s_{{2}}{\s_{{3}}}^{2}-48\,{\s_{{2}}}^{3}\s_{{3}}-1696\,{\s_{{2}}}^{4}-1568\,{\s_{{2}}}^{3}-536\,{\s_{{2}}}^{2}\s_{{3}}  \\
c_{4} & = \s_{{2}} \left( -{\s_{{3}}}^{2}-1180\,{\s_{{2}}}^{2}+16\,{\s_{{2}}}^{3}-376\,\s_{{2}}\s_{{3}} \right)  \\
c_{3} & = 112\,{\s_{{2}}}^{2}-36\,{\s_{{3}}}^{2}-16\,{\s_{{2}}}^{3}-116\,\s_{{2}}\s_{{3}} \\
c_{2} & = 12\,\s_{{2}} \left( 22\,\s_{{2}}-3\,\s_{{3}} \right)  \\
c_{1} & = 28\,\s_{{3}}-16\,{\s_{{2}}}^{2}+112\,\s_{{2}} \\
c_{0} & =  12\,\s_{{2}}\\
\end{split}
\]
Its discriminant has the following factors as in Eq.~\eqref{eq_Z2_3}. 

Each of these components can be expressed in terms of the absolute invariants $t_1, \dots t_6$ as defined in \cite{hyp_mod_3}. Since they are large expressions we do not display them. 

\qed

\noindent The following determines a nice family of curves with automorphism group $\Z_2^3$.

\begin{lemma}
Let $C$ be a genus 3 curve with equation 
\[ y^2 = {\frac {t^4}{256}}\, {x}^{8}+{\frac {t^4}{256}}\, {x}^{6}+ \frac {t^2} {32}\, \left( t+28 \right)  {x}^{4}+ \frac {t^2} {16} \, {x}^{2}+1 \]
such that $t\in \C \setminus \{ -16,  0,  48 \}$. Then, $\Aut (C) \iso \Z_2^3$ and $C$ has $N_r$ $2$-Weierstrass points of weight $r$ as described in the table below.

\begin{center}
$\begin{array}{|c|c|c|c|}
\hline
&N_1&N_2&N_3 \\ \hline
t=-112/3 & 24 & 0 & 12 \\ \hline
t=14\pm14\sqrt{-15} & 16 & 16 & 4 \\ \hline
t\in\mathbb{C}\setminus\{-16,0,48,-112/3,14\pm14\sqrt{-15}\} & 48 & 0 & 4\\ \hline
\end{array}$
\end{center}
\end{lemma}

\proof
Let us assume that the dihedral invariants satisfy the first factor of the Eq.~\eqref{eq_Z2_3}. Since this is a rational curve we can parametrize it as follows
\[ \s_2 = \frac 1 {16} \, t^2 \quad \s_3 = \frac 1 {16} \, ( t + 28) \, t^2 \]
In this case the curve $C$ becomes 
\[ y^2 = {\frac {t^4}{256}}\, {x}^{8}+{\frac {t^4}{256}}\, {x}^{6}+ \frac {t^2} {32}\, \left( t+28 \right)  {x}^{4}+ \frac {t^2} {16} \, {x}^{2}+1 \]
with discriminant $\Delta = t^{28} \, (t-48)^4 \, (t+16)^6 \neq 0$. 
The Wronskian is 
\begin{Small}
\[
\begin{split} 
\Omega_2 & = \frac {x \, (tx^2+4) (tx^2-4)^3}     { ( {t}^{4}{x}^{8}+{t}^{4}{x}^{6}+8\,{t}^{3}{x}^{4}+224\,{t}^{2}{x}^{4}+16\,{t}^{2}{x}^{2}+256)^{9/2}}  \, \left( {t}^{2}{x}^{4}+24\,t{x}^{2} +16 \right) \, ( 3t^8 x^{16}  \\
& -4\,{t}^{6} \left( -16\,t+{t}^{2}-896 \right) x^{14} -16\,{t}^{5} \left( 5\,{t}^{2}+3584+220\,t \right)  x^{12} -192\,{t}^{4} \left( 9\,{t}^{2}+2688+368\,t \right)  x^{10} \\
& -512\,{t}^{3} \left( 487\,t+3584+23\,{t}^{2} \right)  x^8 -3072\,{t}^{2} \left( 9\,{t}^{2}+2688+368\,t \right)  x^6   -4096\,t \left( 5\,{t}^{2}+3584+220\,t \right)  x^4 \\
& + (14680064+262144\,t-16384\,{t}^{2})  x^2 + 196608 )(dx)^{27}
\end{split}
\]
\end{Small}
Hence, the curve has four $2$-Weierstrass points of weight 3 which come from the two roots of the factor $(tx^2-4)^3=0$.  Note that $x=0$ is a root of order 1, so the points $(0,\pm1)$ have weight 1.  Removing these factors as well as the denominator, we obtain a polynomial in $x^2$ which we can write as $$h(x^2)=\Omega_2\cdot\frac{( {t}^{4}{x}^{8}+{t}^{4}{x}^{6}+8\,{t}^{3}{x}^{4}+224\,{t}^{2}{x}^{4}+16\,{t}^{2}{x}^{2}+256)^{9/2}}{x(tx^2-4)^3}$$ for $\deg(h(x))=11$.  We now check $h(x)$ for multiple roots.

One finds that $$\Delta(h,x)=2^{289}3^97^3\cdot t^{93}(16+t)^{14}(3t+112)^6(t-48)^6(t^2-28t+3136)^4.$$  Since we do not consider the cases where $t=0,-16,48$, to make $\Delta(h,x)=0$, we look at $t=-112/3$ and $t=14\pm14\sqrt{-15}.$

When $t=-112/3$, $$h(x)=c(784x^2-504x+9)^3(28x-3)(81 + 168 x + 784 x^2) (3 + 56 x + 2352 x^2)$$ for some constant $c$.  Thus, $h(x)$ has two roots of order 3 and five roots of order 1.  Going back to $\Omega_2$, these roots lead to eight 2-Weierstrass points of weight 3 and twenty 2-Weierstrass points of weight 1.

When $t=14\pm14\sqrt{-15}$, $h(x)$ has four roots of order 2 and 3 roots of order 1.  These lead to sixteen 2-Weierstrass points with weight 2 and twelve 2-Weierstrass points with weight 1.

Note that for any $t\neq0$, the numerator of $\Omega_2$ is a polynomial of degree 29, so the two points at infinity are 2-Weierstrass points with weight 1.

\qed

The other component is also a genus 0 curve and the same method as above can also be used here.

\begin{thm}
The locus in $\H_3$ of curves with full automorphism group $\Z_2^3$ which have $2$-Weierstrass points is a 1-dimensional variety with two irreducible components.  Each component is a rational family. The equation of a generic curve in each family is given in terms of the parameter $t$.  For such curves the field of moduli is a field of definition. 
\end{thm}

\subsection{1-dimensional loci}
There are three cases of groups which correspond to 1-dimensional loci in $\H_3$, namely the groups $\Z_2 \times D_8$, $D_{12}$, and $\Z_2 \times \Z_4$. 


\subsubsection{The case $\Aut (C) \iso \Z_2 \times D_8$.}

\begin{prop}
Let $C$ be a genus 3 hyperelliptic curve with full automorphism group $\Z_2 \times D_8$.  Then $C$ is isomorphic to a curve of the form $y^2=tx^8+tx^4+1$ for some $t\neq0,4.$  For any other $t\neq-140,-980/3$, $C$ has $N_r$ $2$-Weierstrass points of weight $r$ as described in the table below.

\begin{center}
$\begin{array}{|c|c|c|c|}
\hline
&N_1&N_2&N_3 \\ \hline
t=196 & 24 & 0 & 12 \\ \hline
t=-196/5 & 16 & 16 & 4 \\ \hline
t\in\mathbb{C}\setminus\{0,4,-140,-980/3\} & 48 & 0 & 4\\ \hline
\end{array}$
\end{center}

Then, $C$ has at least two $2$-Weierstrass points of weight 3.  Moreover, if   $C$ is isomorphic to   the   curve  $y^2= t x^8 + t x^4 +1$  for $t = 196 $ (resp. $t= - \frac {196} {15}$),  then $C$ has in addition 8 other points of weight 3 (resp.  16 points of weight 2). 
\end{prop}

\proof
In this case the curve has equation $y^2= t x^8 + t x^4 +1$, 
with discriminant $\Delta = 2^{16} \cdot t^7 \, (t-4)^4 \neq 0$, where $t = - 28 \frac {5 t_4 +28} {t_4-4}$; see \cite[Lemma 7]{b-th}.   The Wronskian is as follows
\[ \Omega_2 = 34560\, (t-4) \, {\frac {t{x}^{3}  \left( t{x}^{8}-1 \right)  \left( 7\,{t}^{2}{x}^{16}-18\,{t}^{2}{x}^{12}+3\,{t}^{2}{x}^{8}-98\,t{x}^{8}-18\,t{x}^{4}+7 \right) }{ \left( t{x}^{8}+t{x}^{4}+1 \right) ^{9/2}}}(dx)^{27} \]
Since $x=0$ is of multiplicity 3, then the points $(0, \pm 1)$ have each weight 3. 

The other factors of the Wronskian, namely 
\[ f(x) =  \left( t{x}^{8}-1 \right)  \left( 7\,{t}^{2}{x}^{16}-18\,{t}^{2}{x}^{12}+3\,{t}^{2}{x}^{8}-98\,t{x}^{8}-18\,t{x}^{4}+7 \right)  \]
has double roots if its discriminant is zero.  This happens if  $t = 196$ or $t = - \frac {196} {15}$. 
If $t=196$ then 
\[ \Omega_2 = 9103933440\,{\frac {{x}^{3} \left( 14\,{x}^{4}+1 \right)  \left( 196\,
{x}^{8}-476\,{x}^{4}+1 \right)  \left( 14\,{x}^{4}-1 \right) ^{3}}{
 \left( 196\,{x}^{8}+196\,{x}^{4}+1 \right) ^{9/2}}}(dx)^{27}
\]
Hence, there are 24 points of weight 1, and 8 other  points of weight 3 which come from the roots of $14x^4 =1$.

If $t= - \frac {196} {15}$, then the curve $C$ becomes 
\[ y^2= -{\frac {196}{15}}\,{x}^{8}-{\frac {196}{15}}\,{x}^{4}+1     \]
and the Wronskian
\[
\Omega_2 = -614515507200000\,{\frac {{x}^{3} \left( 15+196\,{x}^{8} \right)  \left( -15-252\,{x}^{4}+196\,{x}^{8} \right) ^{2}}{ \left( -15\, \left( 14\,{x}^{4}+15 \right)  \left( 14\,{x}^{4}-1 \right)  \right) ^{9/2}}(dx)^{27}}
\]
Hence, there are 16 points of weight 1 and 16 points of weight 2.

Finally, observe that since the numerator of $\Omega_2$ is a polynomial in $x$ of degree 27, the two points at infinity have $2$-weight equal to $30-27=3$.
\qed

\subsubsection{The case $\Aut (\X) \iso D_{12}$}

Let $C$ be a genus 3 hyperelliptic curve with full automorphism group $D_{12}$. In this case the curve has equation 
\[ y^2= x \, ( tx^6 + t x^3 +1) \]
for $ t = \frac 7 2 \, \frac {5 t_4 +7} {t_4-2} $  and  discriminant  $\D = 3^6 \cdot t^5 \, (t-4)^3 \neq 0$;  see \cite[Lemma ]{b-th} for details.

In particular, for a curve $C$ given by the equation $y^2=f(x)$, with $\deg(f)=7$, there is one point at infinity, which is singular.  This point is a branch point, and in the desingularization remains as one point, which we will  denote here by $P^{\infty}$.  Let $\{\alpha_i\}$ denote the roots of $f(x)$, and $R_i=(\alpha_i,0)$ the affine branch points.  Let $w\in\mathbb{C}\setminus\{\alpha_i\}$ and let $P^w_1$ and $P^w_2$ denote the points over $w$.  One has the following divisors.

\begin{center}
$\begin{array}{|l|l|} \hline & \\
\div (dx)=\displaystyle\left(\sum_{i=1}^7 R_i\right) - 3P^\infty & 
\div (y)=\displaystyle\left(\sum_{i=1}^7 R_i\right) - 7P^\infty \\ & \\ \hline & \\
\div (x-w) = P_1^w+P_2^w - 2P^\infty & 
\div (x- \alpha_i) = 2R_i - 2P^\infty \\ & \\
 \hline
\end{array}$\end{center}\bigskip

Working with these divisors, one finds that a basis of holomorphic $2$-differentials is given by $$\frac{1}{y^2}(dx)^2\cdot\left\{1,(x-\beta),(x-\beta)^2,(x-\beta)^3,(x-\beta)^4,y\right\}$$ for any $\beta\in\mathbb{C}$.  Letting $\beta=\alpha_i$, the $2$-Weierstrass weight for the branch point $R_i$ is 6.  And using any value of $\beta$, one finds orders of vanishing $8,6,4,2,0,1$ at $P^\infty$, so $w^{(2)}(P^\infty)=6$ as well.

\begin{prop}
Let $C$ be a genus 3 hyperelliptic curve with full automorphism group $D_{12}$.  By \cite[Lemma 8]{b-th}, $C$ has equation $y^2=x(tx^6+tx^3+1)$. Then, $C$ has non-branch points with $2$-Weierstrass weight greater than 1 if and only if $t = - \frac {49} 8$ or $t= \frac {1787} 8 \pm \frac {621} 4 \sqrt{2}$. 

In particular, for each value of $t$, $C$ has $N_r$ $2$-Weierstrass points of weight $r$ as described in the table below.

\begin{center}
$\begin{array}{|c|c|c|c|}
\hline
&N_1&N_2&N_3 \\ \hline
t=-49/8 & 24 & 0 & 12 \\ \hline
t=1787/8\pm621/4\sqrt{2} & 36 & 12 & 0 \\ \hline
t\in\mathbb{C}\setminus\{0,4,-49/8,1787/8\pm621/4\sqrt{2}\} & 60 & 0 & 0\\ \hline
\end{array}$
\end{center}

\end{prop}

\proof  In this case the curve has equation  $y^2= x \, ( tx^6 + t x^3 +1)$ for $ t = \frac 7 2 \, \frac {5 t_4 +7} {t_4-2} $  and  discriminant  $\D = 3^6 \cdot t^5 \, (t-4)^3 \neq 0$; see \cite[Lemma ]{b-th} for details. The Wronskian is 
\[
\begin{split}
\Omega_2 & = -135\,{\frac { \left( t{x}^{6}-1 \right)   } { \left( x \left( t{x}^{6}+t{x}^{3}+1 \right)  \right) ^{9/2}}} \, \left( 7\,{t}^{4}{x}^{24}+28   \,{t}^{4}{x}^{21}-336\,{t}^{4}{x}^{18}+1216\,{t}^{4}{x}^{15}  \right. \\
& \left. -128\,{t}^{4}{x}^{12}+1540\,{t}^{3}{x}^{18}-4668\,{t}^{3}{x}^{15}+6672\,{t}^{3}{x}^{12}+1216\,{t}^{3}{x}^{9}-24150\,{t}^{2}{x}^{12} \right.\\
& \left. -4668\,{t}^{2}{x}^{9}-336\,{t}^{2}{x}^{6}+1540\,t{x}^{6}+28\,t{x}^{3}+7 \right)(dx)^{27} \\
\end{split}
\]
Its discriminant has factors
\[ \Delta (\Omega_2 , x) = t^{145} \, (t-4)^{42} \, \left( 64\,{t}^{2}-28592\,t+108241 \right)^9 \,  \left( 8\,t+49 \right)^{12} \]
Since $t \neq 0, 4$, then the $\Omega_2$ form has multiple roots if and only if 
\[ t = - \frac {49} 8, \quad t= \frac {1787} 8 + \frac {621} 4 \sqrt{2}, \quad t= \frac {1787} 8 - \frac {621} 4 \sqrt{2}. \]
For each one of these values of $t$, the form $\Omega_2$ has multiple zeroes and hence 2-Weierstrass points of weight at least 2.

For $t = - \frac {49} 8$ the numerator of $\Omega_2$ is the polynomial 
\[  \left( 49\,{x}^{6}+8 \right)  \left( 49\,{x}^{6}+616\,{x}^{3}-8 \right)  \left( 49\,{x}^{6}-140\,{x}^{3}-8 \right) ^{3}
\]
which has six roots of multiplicity 3. Hence, the curve $ y^2 = x \left( -{\frac {49}{8}}\,{x}^{6}-{\frac {49}{8}}\,{x}^{3}+1 \right)$ 
has twelve 2-Weierstrass points of weight 3. There are twelve simple roots of the polynomial $\Omega_2 (x)$ and therefore twenty-four points of weight 1. 

For $t= \frac{1787}8 \pm \frac{621}4 \sqrt{2}$, the numerator of $\Omega_2$ is the polynomial
\[ 
 \left(108241 x^6 - (60536 \pm 35532 \sqrt2) x^3 +(14296 \pm 9936 \sqrt2)  \right)^2 g(x)\] for $g(x)$ a degree-18 polynomial with coefficients in $\mathbb{Z}[\sqrt{2}]$ and distinct roots.  The numerator of $\Omega_2$ has six double roots which lead to twelve 2-Weierstrass points of weight 2.  The remaining eighteen roots are single roots, leading to thirty-six 2-Weierstrass points of weight 1.

Finally, note that in both cases, the 2-Weierstrass points we have calculated make a contribution of 60 to the total weight.   The eight branch points (including the point at infinity) each have 2-Weierstrass weight 6, thus making a contribution of 48 to the total weight, which is 108.

\qed

\begin{rem}
Notice that in the case of the curve  $ y^2 = x \left( -{\frac {49}{8}}\,{x}^{6}-{\frac {49}{8}}\,{x}^{3}+1 \right)$, even though the curve is defined over $\Q$ the $2$-Weierstrass points are defined over a degree 6 extension of $\Q$.  For more details on the field of definition of $q$-Weierstrass points see \cite{silverman}. 
\end{rem}

\subsubsection{The case $\Aut (\X) \iso \Z_2 \times \Z_4$}   

\begin{prop}
Let $C$ be a genus 3 hyperelliptic curve with full automorphism group  $\Z_2 \times \Z_4$.  
Then, $C$ has $2$-Weierstrass points if and only if $C$ is isomorphic to one of the   curves  $y^2=\left( t x^4-1 \right) \, \left( t x^4 + tx^2+1\right)  $, for $t = -  8$ or it is a root of  
\begin{small}
\begin{equation}
\begin{split}
  \, \, {t}^{8}+600822\,{t}^{7}+71378609\,{t}^{6}+4219381768\,{t}^{5}+85080645104\,{t}^{4}   -2272444082944\,{t}^{3} \\
   +16480136388352\,{t}^{2}-50330309965824\,t+56693912375296  & = 0
\end{split}
\end{equation} 
\end{small}
In the first case, the curve has two  $2$-Weierstrass points of weight 3. 
\end{prop}

\proof
The equation of this curve is given by
\[ y^2 = \left( t x^4-1 \right) \, \left( t x^4 + tx^2+1\right) \]
with discriminant $\D - 2^{12} \cdot t^{14} (t-4)^6$. 
The numerator of the Wronskian is a degree 29 polynomial in $x$, given by $x \phi(x)$, where   
\begin{small}
\[
\begin{split}
\phi(x) & =  \left( 24\,{t}^{7}-3\,{t}^{8} \right) {x}^{28}+ \left( -4\,{t}^{7}+4\,{t}^{8}+224\,{t}^{6} \right) {x}^{26}+ \left( 63\,{t}^{7}+504\,{t}^{6} \right) {x}^{24}+1368\,{t}^{6}{x}^{22} \\
& + \left( 4\,{t}^{7}+2888\,{t}^{5}+1045\,{t}^{6} \right) {x}^{20}+ \left( 3360\,{t}^{4}+588\,{t}^{6}+3780\,{t}^{5} \right) {x}^{18}+  ( 3375\,{t}^{5}+108\,{t}^{6}\\
& +5544\,{t}^{4} ) {x}^{16} + \left( 7632\,{t}^{4}+1056\,{t}^{5} \right) {x}^{14}+ \left( 5544\,{t}^{3}+3375\,{t}^{4}+108\,{t}^{5} \right) {x}^{12}+ ( 3780\,{t}^{3} \\
& +3360\,{t}^{2}+588\,{t}^{4} ) {x}^{10}+ \left( 1045\,{t}^{3}+4\,{t}^{4}+2888\,{t}^{2} \right) {x}^{8}+1368\,{t}^{2}{x}^{6}+ \left( 504\,t+63\,{t}^{2} \right) {x}^{4} \\
& + \left( -4\,t+4\,{t}^{2}+224 \right) {x}^{2}+24-3\,t
\end{split} 
\]
\end{small}
Its discriminant is 
\begin{small}
\[ 
\begin{split}
\Delta = & t^{275} \, (t-4)^{90} \, \left( t-8 \right)^4 \left( {t}^{8}+600822\,{t}^{7}+71378609\,{t}^{6}+4219381768\,{t}^{5}+85080645104\,{t}^{4}  \right.    \\
& \left.   -2272444082944\,{t}^{3}+16480136388352\,{t}^{2}-50330309965824\,t+56693912375296 \right)^4   \\
\end{split}
\] 
\end{small}
Hence, for $t=8$ or $t$ satisfying the octavic polynomial 
the corresponding curve has $2$-Weierstrass points.  
In the first case, $t=8$, the curve becomes
\[ y^2 = \left( 8\,{x}^{4}-1 \right)  \left( 8\,{x}^{4}+8\,{x}^{2}+1 \right)  \]
The Wronskian $\Omega_2$ has $x=0$ as a triple root. Hence, the points $(0, i)$ and $(0, -i)$, for $i^2=-1$ are $2$-Weierstrass points of weight 3.

If $t$ is a root of the second factor, then the Galois group of this octavic is $S_8$ and therefore not solvable by radicals. 

\qed

\noindent Summarizing we have the following theorem. 
\begin{thm}[Main Theorem]
Let $C$ be a genus 3 hyperelliptic curve such that $|\Aut (C) | > 4$ and $\H (\Aut (C) )$ is a locus of dimension $d> 0$ in $\H_3$ and $\pi : C \to \mathbb P^1$ the hyperelliptic projection. Then, 
each branch point of $\pi$ has $2$-weight 6 and  one of the following holds:

i) If $\Aut (C) \iso \Z_2^3$, then $C$ has non-branch $2$-Weierstrass points of weight greater than one if and only if its corresponding dihedral invariants $\s_2, \s_3, \s_4$ satisfy Eq.~\eqref{eq_Z2_3}

ii) If $\Aut (C) \iso \Z_2 \times D_8$ then $C$ has at least four non-branch $2$-Weierstrass points of weight 3.  Moreover, if   $C$ is isomorphic to the curve  $y^2= t x^8 + t x^4 +1$, for $t = 196 $ (resp. $t= - \frac {196} {15}$) then $C$ has in addition 8 other points of weight 3 (resp.  16 points of weight 2). 

iii) If $\Aut (C) \iso D_{12}$ then $C$ has non-branch $2$-Weierstrass points with weight greater than one if and only if $C$ is isomorphic to one of the   curves  $y^2= x (tx^6+tx^3+1)$, for $t = - \frac {49} 8$ or $t= \frac {1728} 8 \pm \frac {621} 4 \sqrt{2}$. 

In the first case, the curve has twelve  $2$-Weierstrass points of weight 3 and in the other two cases twelve $2$-Weierstrass points of weight 2.

iv) If $\Aut (C) \iso \Z_2 \times \Z_4$ then  $C$ has $2$-Weierstrass points if and only if $C$ is isomorphic to one of the   curves  $y^2=\left( t x^4-1 \right) \, \left( t x^4 + tx^2+1\right)  $, for $t = -  8$ or it is a root of  
\begin{Small}
\begin{equation}
\begin{split}
  \, \, {t}^{8}+600822\,{t}^{7}+71378609\,{t}^{6}+4219381768\,{t}^{5}+85080645104\,{t}^{4}   -2272444082944\,{t}^{3} \\
   +16480136388352\,{t}^{2}-50330309965824\,t+56693912375296  & = 0
\end{split}
\end{equation} 
\end{Small}
In the first case, the curve has two  $2$-Weierstrass points of weight 3. 
\end{thm}

\section{Further directions}
In this paper we explicitly determined the $2$-Weierstrass points of genus 3 hyperelliptic curves with extra automorphisms.  Similar methods can be used for $3$-Weierstrass points even though the computations are longer and more difficult. 

In each case, the curves which have $2$-Weierstrass points are determined uniquely. That follows from the fact that parameters $\s_2, \s_3, \s_4$ or the parameter $t$ are rational functions in terms of the absolute invariants $t_1, \dots , t_6$. The above results also provide a convienient way to check if a genus 3 hyperelliptic curve has $2$-Weierstrass points. This is done by simply computing the absolute invartiants $t_1, \dots , t_6$ and checking which locus they satisfy. It is our goal to incorporate such methods in a computational package for genus 3 curves. 

In \cite{SS2} we intend to study generalizations of such results to all hyperelliptic curves.  The dihedral invariants used here are generalized to all hyperelliptic curves in \cite{g-sh}.


\bibliographystyle{amsplain}


\end{document}